\documentclass{amsart}
\usepackage[utf8]{inputenc}
\usepackage[utf8]{inputenc}
\usepackage[utf8]{inputenc}
\usepackage[english]{babel}
\usepackage{amssymb}
\usepackage{amsfonts}
\usepackage{amsmath}
\usepackage{amscd}
\usepackage{amsthm}
\usepackage{geometry} 
\usepackage{latexsym}
\usepackage{epsfig}
\usepackage{graphicx}
\usepackage{mathtools}
\usepackage{xcolor}
\usepackage[all]{xypic}

\newtheorem{theorem}{Theorem} [section]

\newtheorem{definition}[theorem]{Definition}
\newtheorem{proposition}[theorem]{Proposition}
\newtheorem{lemma}[theorem]{Lemma}

\theoremstyle{definition}
\newtheorem{example}[theorem]{Example}
\newtheorem{remark}[theorem]{Remark}

\title[The HH condition and the AOP for surfaces]{The Harbourne-Hirschowitz condition and the anticanonical orthogonal property for surfaces}

\author{Abel Castorena}
\address{Abel Castorena. Centro de Ciencias Matem\'aticas. Universidad Nacional Aut\'onoma de M\'exico, Campus Morelia. Antigua Carretera a P\'atzcuaro 8701, Col. Ex-Hacienda San Jos\'e de la Huerta, C.P. 58089. Morelia, Mich., Mexico}
\email{abel@matmor.unam.mx}

\author{Juan Bosco Fr\'ias-Medina}
\address{Juan Bosco Fr\'ias Medina. Instituto de F\'isica y Matem\'aticas. Universidad Michoacana de San Nicol\'as de Hidalgo. Edificio C-3, Ciudad Universitaria. Avenida Francisco J. M\'ugica s/n, Colonia Felicitas del R\'io,
C.P. 58040,  Morelia, Mich., Mexico}
\email{juan.frias@umich.mx}

\date{}

\thanks{The first author is supported by Grants PAPIIT UNAM IN100419 ``Aspectos Geom\'etricos del moduli de curvas $M_g$" and CONACyT, M\'exico A1-S-9029 ``Moduli de curvas y curvatura en $A_g$". The second author acknowledges the support of Programa de Becas Posdoctorales 2019, DGAPA, UNAM during 2021. The second author was supported by ``Programa de Estancias Posdoctorales por M\'exico Convocatoria 2021 de CONACYT" during the revisions of this paper.}

\begin{document}

\begin{abstract}
In this paper we give the first steps toward the study of the Harbourne-Hirschowitz condition and the Anticanonical Orthogonal Property for regular surfaces. To do so, we consider the Kodaira dimension of the surfaces and study the cases based on the Enriques-Kodaira classification.
\end{abstract}

\maketitle


\section{Introduction}

In 1966, D. Mumford in his book \textit{Lectures on curves on an Algebraic Surface} \cite{Mum66} proposed four lines of study for curves on a general surface. He called ``The problem of Riemann-Roch'' to the first of these problems and stated the following: given a curve $C$ on a surface $S$, determine the dimension of the complete linear system of curves containing $C$. Mumford pointed out that this problem is equivalent to the problem of computing the dimension of the zeroth cohomology group associated with the line bundle $\mathcal{O}_S(C)$ on $S$. To do so, one can use the Riemann-Roch theorem: for the line bundle $\mathcal L:=\mathcal{O}_S(D)$ associated with a divisor $D$ on a surface $S$, we have that
\begin{equation*}
    h^0(S,\mathcal{O}_S(D))-h^1(S,\mathcal{O}_S(D))+h^2(S,\mathcal{O}_S(D))
    = p_a(S)+ 1 + \frac{1}{2}(D^2-K_S\cdot D),
\end{equation*}
here, for $i=0,1,2$, $h^i(S,\mathcal{O}_S(D))$ denotes the dimension of the $i$-th cohomology group $H^i(S,\mathcal{L})$, $p_a(S)$ denotes the arithmetic genus of $S$ and $K_S$ denotes a canonical divisor on $S$. So, to compute explicitly the dimension of $H^0(S,\mathcal{O}_S(D))$, one has to compute the dimension of the other cohomology groups and the arithmetic genus of the surface. 

It is well-known (see for example, \cite[Chapter VI]{BHPV}) that complex algebraic surfaces are classified by the Enriques-Kodaira classification depending on their Kodaira dimension. Indeed, if $S$ is a minimal algebraic surface and $\kappa(S)$ denotes its Kodaira dimension, then one of the following occurs:
\begin{enumerate}
    \item[a)] $\kappa(S)=-\infty$, so $S$ is a minimal rational surface or a ruled surface over a curve of positive genus (recall that the minimal surfaces of class VII are never algebraic).
    \item[b)] $\kappa(S)=0$, so $S$ is one of the following: an Enriques surface, a bielliptic surface, a Kodaira surface (primary or secondary), a $K3$ surface, or a tori.   
    \item[c)] $\kappa(S)=1$, so $S$ is a minimal properly elliptic fibration.
    \item[d)] $\kappa(S)=2$, so $S$ is a minimal surface of general type.
\end{enumerate}

In \cite{Har96}, B. Harbourne considered the following situation in the case of rational surfaces: let $S=S_n\to S_{n-1}\to\cdots\to S_1\to S_0=\mathbb{P}^2$ a composition of morphisms, where $S_i\to S_{i-1}$ is the blow-up at a point $p_i\in S_{i-1}$ for $i=1,\dots, n$. Note that the $p_i$'s could be infinitely near points. We say that the points $\{p_i\}_{i=1}^n$ are in \textit{good position} if the surface $S$ obtained by blow-ups the points $p_1,\dots,p_n$ has no irregular effective nef divisor classes, that is, if $D$ is an effective nef divisor on $S$, then $h^1(S,\mathcal{O}_S(D))=0$. Since for a rational surface $S$ we have that $p_a(S)=0$, the property of having points in good position implies that the dimension of $H^0(S,\mathcal{O}_S(D))$ can be computed for any divisor class on $S$ (see for example \cite[Theorem~3]{FL20}) and so, one can solve the problem of Riemann-Roch in this case. Nowadays, it is conjectured that this fact hold if $S$ is a blow-up of $\mathbb{P}^2$ at ``sufficiently general points". Such conjecture was stated in equivalent forms by Harbourne in \cite{Har86}, A. Hirschowitz in \cite{Hir89}, B. Segre in \cite{Seg61} and A. Gimigliano in \cite{Gim89}, and it is known as the \textit{Harbourne-Hirschowitz conjecture} or \textit{Segre-Harbourne-Gimigliano-Hirschowitz conjecture}.  

Another important problem involving the first cohomology group appears in the context of the deformation theory of Hilbert schemes of divisors on surfaces. For an effective divisor $D_0$ on a surface $S$, we denote by $H_{D_0,S}$ the Hilbert scheme of all effective divisors on $S$  that are algebraically equivalent to $D_0$. The obstruction space to the deformation theory $H_{D_0,S}$ is induced by the long exact sequence in cohomology associated with the sequence
\begin{equation*}
 0\to \mathcal O_S\to \mathcal O_S(D_0) \to\mathcal O_{D_0}(D_0)\to 0. 
\end{equation*}
More precisely, the obstruction space of the functor $H_{D_0,S}$ is given by 
\begin{equation*}
    K_1= \text{Im}\big( H^1(S,\mathcal{O}_S(D_0)) \rightarrow H^1(D_0,\mathcal{O}_{D_0}(D_0)) \big).
\end{equation*}
The divisor $D_0$ is \textit{semiregular} if $K_1=0$. In such case, we have that $H_{D_0,S}$ is scheme-theoretically smooth at $D_0$ of dimension equal to $h^0(D_0,\mathcal{O}_{D_0}(D_0))$. See \cite{FM99} for all the details in this construction. In addition, in the context of \textit{rigid divisors} \cite{HP20} it is relevant the study of some first cohomology groups. 

Motivated by the concept of points in good position, the second author along with M. Lahyane introduced in \cite{FL18} a natural generalization of this notion for any surface: 

\begin{definition}
A smooth projective surface $S$ is a \textbf{Harbourne-Hirschowitz surface} (HH surface for short) if for every effective nef divisor $D$ on $S$ we have  that $h^1(S,\mathcal O_S(D))=0$.
\end{definition} 

We will also refer to the requirement of the above definition as the \textit{Harbourne-Hirschowitz condition}. Note that in the case of dealing with a Harbourne-Hirschowitz surface $S$, we always have the semiregular condition for every effective nef divisor on $S$.

Along with the concept of Harbourne-Hirschowitz surfaces, in the same work \cite{FL18} the following notion was introduced:

\begin{definition}
A smooth projective surface $S$ satisfies the \textbf{Anticanonical Orthogonal Property} (AOP for short) if for every nef divisor $D$ on $S$, the equality $-K_S\cdot D=0$ implies that $D=0$.
\end{definition}

Similarly, we refer to the requirement of this definition as the \textit{AOP condition}. This concept was introduced in general but was used in the context of anticanonical rational surfaces, indeed, the interest of having an anticanonical rational surface $S$ satisfying the AOP condition is that such surface will be a Harbourne-Hirschowitz surface (\cite[Theorem~2.5]{FL18}).

The aim of this work is to give the first steps toward the study of the Harbourne-Hirschowitz surface and the Anticanonical Orthogonal Property based on the Enriques-Kodaira classification. We restrict ourselves to the case of regular surfaces, this will be justified in the next section. The paper is organized as follows. In Section~\ref{hhaop} we review some properties in the general context of the Harbourne-Hirschowitz surfaces and the surfaces that satisfy the Anticanonical Orthogonal Property. In Sections~\ref{dim2}, \ref{dim1} and \ref{dim0} we study these properties for surfaces whose Kodaira dimension is equal to $2$, $1$ and $0$ respectively. Finally, in Section~\ref{rational} we study the case of surfaces of Kodaira dimension $-\infty$, and more precisely, the rational surfaces. This case is where the Harbourne-Hirschowitz surfaces and Anticanonical Ortogonal Property are more natural and  interesting to study.

\section{Harbourne–Hirschowitz surfaces and the anticanonical orthogonal property}\label{hhaop}

In this section we review some properties of Harbourne-Hirschowitz surfaces and surfaces which satisfies the Anticanonical Orthogonal Property. Throughout this paper will work over the field of complex numbers $\mathbb{C}$ and we will assume that all the surfaces are smooth.

Let $S$ be a HH surface and let $D$ be any effective nef divisor on $S$. Consider the following exact sequence:
\begin{equation*}
 0\to \mathcal O_S(-D)\to \mathcal O_S\to\mathcal O_D\to 0.   
\end{equation*}
Since $h^1(S,\mathcal O_S(D))=0=h^1(S,\mathcal{O}_S(K_S-D))$, from cohomology one has the following sequence:
\begin{equation*}
 0\to H^0(S,\mathcal{O}_S(K_S-D))\to H^0(S,\mathcal{O}_S(K_S))\to H^0(D,\mathcal{O}_D(K_S))\to 0.   
\end{equation*}
\noindent Since $p_g(S)=h^2(S,\mathcal O_S)=h^0(S,\mathcal{O}_S(K_S))$, then \begin{equation*}
  p_g(S)=h^0(S,\mathcal{O}_S(K_S-D))+h^0(D,\mathcal{O}_D(K_S)).  
\end{equation*}
So, on a Harbourne-Hirschowitz surface we have  that the geometric genus $p_g(S)$ can be written as the sum of the dimensions of the cohomology groups  
$H^0(S,\mathcal{O}_S(K_S-D))$ and $H^0(D,\mathcal{O}_D(K_S))$.

\begin{remark}\label{remark-hhregular}
If $S$ is a HH surface, then $S$ is a regular surface. Indeed, since the zero divisor $D=0$ is nef and effective, the hypothesis of being a HH surface implies that $h^1(S,\mathcal{O}_S(D))=h^1(S,\mathcal{O}_S)=0$. Therefore, $S$ is regular.
\end{remark}

Note that both HH and AOP conditions cannot be ensured by blow-up a surface which satisfy these properties. For example, consider a rational surface $S$ with $K_S^2=1$, this surface satisfies the HH and AOP conditions (see Proposition~\ref{k2pos} below). If we blow-up such surface at a point, we will obtain a rational surface $S'$ with $K_{S'}^2=0$ and in general, we may lose the fulfillment of the HH and AOP conditions (see Theorem~\ref{k2zero} below). However, below we prove that both conditions could be preserved under certain morphisms if the domain satisfies the properties.

\begin{proposition}\label{contrachh}
Let $\pi:S\rightarrow T$ be the blow-up of $T$ at a point. If $S$ is a HH surface, then $T$ also is a HH surface. In particular, the result holds if $\pi$ is a birational morphism. 
\end{proposition}
\begin{proof}
Let $D$ be an effective nef divisor on $T$. The induced application $\pi^*:\mathrm{Pic}(T)\rightarrow\mathrm{Pic}(S)$ preserves the effectiveness and nefness of a divisor, then we have that $D'=\pi^*(D)$ is an effective nef divisor on $S$. Since $S$ is a HH surface, then $h^1(S,\mathcal{O}_S(D'))=0$.

On the other hand, since the dimension of the cohomology groups are preserved under $\pi^*$ (see for example \cite[Lemma 3 (b)]{Har96}), we have the equality $h^1(T,\mathcal{O}_T(D))=h^1(S, \mathcal{O}_S(D'))=0$. Therefore, $T$ is a HH surface.
\end{proof}

\begin{proposition}
Let $\pi:S\rightarrow T$ be a dominant morphism such that contracts the ramification divisor to points. If $S$ satisfies the AOP condition, then $T$ also satisfies the AOP condition. 
\end{proposition}

\begin{proof}
Denote by $R$ the ramification divisor associated with $\pi$. Let $D$ be a nef divisor on $T$ such that $-K_T\cdot D=0$. Note that $\pi^*(D)$ is a nef divisor on $S$. Using the fact that $-K_S=\pi^*(-K_T)-R$,  the projection formula and the hypothesis that the morphism $\pi$ contracts $R$ to points, we have that
\begin{equation*}
    -K_S\cdot \pi^*(D)
    =\pi_*(-K_S)\cdot D
    = \pi_*\big(\pi^*(-K_T)-R\big)\cdot D
    =-K_T\cdot D + \pi_*(R)\cdot D 
    =0.
\end{equation*}
The hypothesis implies that $\pi^*(D)=0$ and since $\pi^*$ is an injective application, we conclude that $D=0$.
\end{proof}

\begin{remark}
If $\pi:S\rightarrow T$ is the blow-up of $T$ at a point or more general, a birational morphism, and if $S$ satisfies the AOP condition, then $T$ also satisfies the AOP condition. This result was previously proved in \cite{FL20}.
\end{remark}

\begin{remark}
The adjuction formula also gives a criterion that ensures when a surface $S$ does not satisfy the AOP condition: if there exists a nonsingular curve such that $C^2=0$ and $p_a(C)=1$, then $S$ does not satisfy AOP.
\end{remark}

Note that Remark \ref{remark-hhregular} tells that regular surfaces are the correct ones to study the HH condition and \textit{a priori} the regularity of the surface is not related with the AOP condition. The original motivation of this study came from the study of rational surfaces, they are always regular surfaces and we have the property that being numerically trivial implies that we are dealing with the zero divisor. In particular, such property enables one to study the AOP condition using divisor classes on the N\'eron-Severi group instead of study directly the divisors. In order to follow this original motivation, we restrict ourselves to the case of regular surfaces although the regularity of a surface does not imply that the numerical and linear equivalences coincide.

\section{Kodaira dimension $2$}\label{dim2}

Let $S$ be a minimal surface of general type. One has that $K_S$ is nef and $K^2_S\geq 1$. From Proposition 1 in \cite{Bo} we have that if $C$ is an irreducible curve on $S$ then $K_S\cdot C\geq 0$, and if $K_S\cdot C=0$, then $C^2=-2$ and $C$ is smooth and rational. Moreover the set of curves $E$ with $K_S\cdot E=0$ form a finite set and they are numerically independents on $S$. From this result we have in particular that if $S$ is regular and has no torsion, then it satisfies the AOP condition. 

\begin{proposition}
Let $S$ be a minimal surface of general type with $q=0$. If $S$ has no torsion, then $S$ satisfies the AOP condition.
\end{proposition}
\begin{proof}
Let $D$ be a nef divisor on $S$ such that $K_S\cdot D=0$. Note that $D^2\geq 0$ and that $|mK_S|\neq\emptyset$ for a sufficiently large $m$. Consider the decomposition of $|mK_S|$ in its mobile part $|H|$ and its fixed part $|F|$:
\begin{equation*}
    |mK_S|=|H|+|F|.
\end{equation*}
Using the fact that $S$ is a surface of general type we have that $H^2>0$ for sufficiently large $m$. On the other hand, the hypothesis $D$ nef implies
\begin{equation*}
    D\cdot H= D\cdot (mK_S-F)=mK_S\cdot D-D\cdot F=-D\cdot F\leq 0.
\end{equation*}
Thus, $D\cdot H\geq 0$ and we have that $D\cdot H=0$. Since $H^2>0$ and $D\cdot H=0$, Hodge Index Theorem implies $D^2\leq 0$ and since $D^2\geq 0$, it follows that $D$ is numerically trivial. The condition that $S$ has no torsion implies that $\mathrm{Pic}(S)$ has no torsion. Then, since the torsion subgroup consists of numerically trivial classes, we have that the only numerically trivial class is the trivial one. Therefore, we conclude that $D$ is the zero divisor.
\end{proof}

The study of surfaces of general type is an active field of study nowadays and we are not aware of an example of a regular surface of general type that satisfies the HH condition. We will study concrete cases in a future work. 

On the other hand, recall that one of the importance of the HH condition is to compute explicitly the dimension of the complete linear systems on a surface. In the case of surfaces of general type with $q=0$, the vanishing of the first cohomology group may be not enough to make such computation. By Serre duality we have that $h^2(S,\mathcal{O}_S(D))=h^0(S,\mathcal{O}_S(K_S-D))$. At the same time, since $K_S$ is a non-zero nef divisor, it may occur that the latter cohomology group does not vanish, for example, there exists a minimal surfaces of general type $S$ with $p_g=3$, $q=0$ and such that $|K_S|$ has a non-trivial fixed part (see \cite[Theorem~4]{Bin21}).


\section{Kodaira dimension $1$}\label{dim1}


Recall that if $S$ is a surface whose Kodaira dimension is equal to 1, then the canonical divisor $K_S$ is nef and $K_S^2=0$. Since there exists a large enough $n$ such that $nK_S$ is non-trivial, an immediate consequence of this is the following:

\begin{theorem}
Let $S$ be a surface whose Kodaira dimension is equal to one. Then, $S$ does not satisfy the AOP condition.
\end{theorem}

On the other hand, since every surface whose Kodaira dimension is equal to 1 is an elliptic surface, we prove that the HH condition does not hold.

\begin{theorem}
Let $S$ be a regular surface whose Kodaira dimension is equal to one. Then, $S$ does not satisfy the HH condition.
\end{theorem}
\begin{proof}
In \cite[Section 4]{FM99}, Friedman and Morgan characterized the divisors $D$ with the property $h^1(S,\mathcal{O}_S(D))=0$ in the case of a regular elliptic surface. Such divisors are numerically equivalent to $\frac{1-r}{2}K_S$ for some rational number $r\leq 1$. In particular, the condition $D\cdot K_S=0$ should be satisfied. Since the latter conditions is not satisfied in general, we concluded that the HH condition does not hold.
\end{proof}

In spite of the above results, Friedman and Morgan computed in \cite[Lemma 4.1]{FM99} the dimensions of the cohomology groups for regular elliptic surfaces. In fact, one can note that the dimension of the first cohomology group can be different from zero.

\begin{theorem}[Friedman-Morgan]\label{cohomellip}
Let $S$ be a regular elliptic surface and let $D=af+\sum_{i} b_i F_i$ where $a\geq 0$ and $0\leq b_i\leq m_i-1$. Here, $f$ is the divisor class of a general fiber, the $F_i$'s denote the multiple fibers and the $m_i$ denotes the multiplicity of $F_i$. Then
\begin{align*}
  &  h^0(S,\mathcal{O}_S(D))=a+1; & &\\
    & h^1(S,\mathcal{O}_S(D))=\begin{cases}
    0 & \text{ if } a\leq p_g \\
    a-p_g & \text{ if } a>p_g
    \end{cases}; 
    &  h^2(S,\mathcal{O}_S(D))=\begin{cases}
    p_g-a & \text{ if } a\leq p_g \\
    0 & \text{ if } a>p_g
    \end{cases}. &
\end{align*}
\end{theorem}


\section{Kodaira dimension $0$}\label{dim0}

The regular surfaces with Kodaira dimension zero are the $K3$ surfaces and Enriques surfaces. In both classes of surfaces, the canonical divisor is nef. Firstly, we deal with the AOP condition.

\begin{theorem}
Let $S$ be a $K3$ or an Enriques surface. Then, $S$ does not satisfy the AOP condition.
\end{theorem}
\begin{proof}
Assume that $S$ is a $K3$ surface. In this case, the canonical divisor is trivial. So, if $H$ is an ample divisor then $H\cdot K_S=0$ but $H\neq 0$.

Now, assume that $S$ is an Enriques surface. So, the canonical divisor $K_S$ is a non-trivial nef divisor and satisfies $K_S^2=0$.
\end{proof}

Next, we review the HH condition. Knutsen and Lopez gave in \cite{KL07} a vanishing theorem for $H^1(S,\mathcal{L})$, where $\mathcal{L}$ is a line bundle on $S$, that gives necessarily and sufficient conditions for $K3$ and Enriques surfaces:

\begin{theorem}[Knutsen-Lopez]\label{K3E-KL}
Let $S$ be a $K3$ or an Enriques surface and let $\mathcal{L}$ be a line bundle on $S$ such that $\mathcal{L}>0$ and $\mathcal{L}^2\geq 0$. Then $H^1(S,\mathcal{L})\neq 0$ if and only if one of the following occurs:
\begin{enumerate}
    \item[(i)] $\mathcal{L}\sim nE$ for $E>0$ nef and primitive with $E^2=0$, $n\geq 2$ and $h^1(S,\mathcal{L})=n-1$ if $S$ is a $K3$ surface, $h^1(S,\mathcal{L})=\lfloor \frac{n}{2} \rfloor$ if $S$ is an Enriques surface;
    \item[(ii)] $\mathcal{L}\sim nE+K_S$ for $E>0$ nef and primitive with $E^2=0$, $S$ is an Enriques surface, $n\geq 3$ and $h^1(S,\mathcal{L})=\lfloor \frac{n-1}{2} \rfloor$; 
    \item[(iii)] there is a divisor $\Delta>0$ such that $\Delta^2=-2$ and $\Delta\cdot \mathcal{L}\leq -2$.
\end{enumerate}
 \end{theorem}

Using the above result, we will provide a classification for the $K3$ surfaces that satisfy the HH condition and we will prove that Enriques surfaces never satisfy the HH condition.

\begin{theorem}
Let $S$ be a $K3$. The following statements are equivalent:
\begin{enumerate}
    \item $S$ does not satisfy the HH condition.
    \item There exists a non-trivial effective nef divisor $D$ such that $D^2=0$. 
    \item $S$ admits an elliptic fibration.
\end{enumerate}
\end{theorem}
\begin{proof}
 $(1)\Rightarrow(2)$.
 Assume that $S$ does not satisfy HH. So, there exists a non-trivial effective nef divisor $D$ such that $H^1(S,\mathcal{O}_S(D))\neq 0$. Since $D$ is nef, by Theorem \ref{K3E-KL} there exists  a nef and primitive divisor $E>0$ such that $E^2=0$ and $D\sim nE$ for some $n\geq 2$. In particular, $D^2=0$.

 $(2)\Rightarrow(3)$. 
 Let $D$ be a non-trivial effective nef divisor $D^2=0$. By \cite[Proposition \textbf{2}.3.10]{Huy16}, there exists a smooth irreducible elliptic curve $E$ such that $D\sim mE$, for some $m>0$. Such curve $E$ induces the elliptic fibration.
 
 $(3)\Rightarrow(1)$.
 Assume that $S$ admits an elliptic fibration. Then, by \cite[Proposition \textbf{11}.1.3]{Huy16} there exists a non-trivial divisor $L$ such that $L^2=0$. Moreover,  there exists a non-trivial nef divisor $D$ such that $D^2=0$ by \cite[Remark \textbf{8}.2.13]{Huy16}. Therefore, by \cite[Proposition \textbf{2}.3.10]{Huy16} there exists a smooth irreducible elliptic curve $E$ such that $D\sim mE$, for some $m>0$. Finally, since a smooth irreducible elliptic curve is primitive, we are able to construct non-trivial nef divisors as in (i) of Theorem \ref{K3E-KL}. 
\end{proof}

\begin{example}
The Fermat quartic $S$ in $\mathbb{P}^3$ given by the equation \begin{equation*}
    x_0^4+x_1^4+x_2^4+x_3^4=0
\end{equation*}
is not a HH surface. Indeed, one can take the line $\ell\subset S$ given by $x_1=\xi x_0$ and $x_3=\xi x_2$, where $\xi$ is a primitive eight root of the unity. Projecting $S$ with center $\ell$ onto a disjoint line in $\mathbb{P}^3$, one can construct an elliptic fibration explicitly (see \cite[Example~\textbf{2}.3.11]{Huy16}).
\end{example}

\begin{theorem}\label{dimh0}
Let $S$ be a $K3$ Harbourne-Hirschowitz surface. Let $D$ be a divisor on $S$ and write $D=M+F$, where $M$ is the mobile part of $D$ and $F$ is the fixed part of $D$. Then,
\begin{equation*}
    h^0(S,\mathcal{O}_S(D))=\begin{cases}
    0 & \text{ if } D \text{ is not effective, }\\
    1 & \text{ if } M=0, \\
    2+\dfrac{M^2}{2} & \text{ otherwise. }
    \end{cases}
\end{equation*}
\end{theorem}
\begin{proof}
If $D$ is not effective, then it is clear that $\mathcal{O}_S(D)$ has no global sections and consequently $h^0(S,\mathcal{O}_S(D))=0$. If $D$ does not have mobile part, then $D=F$ and then $h^0(S,\mathcal{O}_S(D))=1$. 

Now, assume that $D$ is an effective divisor and $M\neq 0$. We have that $M$ is a non-zero effective nef divisor such that $h^0(S,\mathcal{O}_S(D))=h^0(S,\mathcal{O}_S(M))$, so lets compute the dimension of the space of global sections associated with $M$. Since $K_S$ is trivial we have that $h^2(S,\mathcal{O}_S(D))=h^0(S,\mathcal{O}_S(-D))=0$ and using the hypothesis that $S$ is Harbourne-Hirschowitz we have that $h^1(S,\mathcal{O}_S(D))=0$. We conclude the result by Riemann-Roch theorem.
\end{proof}

\begin{theorem}
Let $S$ be an Enriques surface. Then, $S$ does not satify the HH condition.
\end{theorem}
\begin{proof}
Since every Enriques surface admits an elliptic fibration (see for example \cite[Theorem 17.5]{BHPV} or \cite[Theorem 10.17]{Bad}), we can consider a smooth irreducible elliptic curve $E$ on $S$. So, we can construct non-trivial nef divisors as in (i) or (ii) of Theorem \ref{K3E-KL} such that the dimension of their first cohomology groups is non-zero.  
\end{proof}

In spite of the above result, there are nef divisors on an Enriques surfaces that satisfy the vanishing of the first cohomology group: if $D$ is a nef divisor such that $D^2\geq 2$, then $h^1(S,\mathcal{O}_S(D))=0$ (see \cite[Section~4]{Dol16}). On the other hand, since Enriques surfaces are elliptic surfaces of genus zero, the dimension of the cohomology groups can be calculated explicitly, see Theorem \ref{cohomellip} in the previous section.


\section{Kodaira dimension $-\infty$}\label{rational}
Surfaces with Kodaira dimension equal to $-\infty$ are the rational surfaces and the ruled surfaces.  
Consider a ruled surface $S$ over a curve $C$ of genus $g$. In this case, the irregularity of $S$ depends on the genus of the curve $C$ (see \cite[Proposition III.21]{Bea96}):
\begin{equation*}
    h^1(S,\mathcal{O}_S)=g(C).
\end{equation*}
So, if $S$ is a regular ruled surface, then necessarily $g(C)=0$ and $S$ is a Hirzebruch surface. Therefore, the only possible regular surfaces with Kodaira dimension equal to $-\infty$ are the rational surfaces. Therefore, from now on we focus our attention on the rational surfaces.

Let $S$ be a rational surface. For convenience of the reader, we remind some notions that will be considered in this section. A divisor $D$ on $S$ is \textit{regular} if $H^1(S,\mathcal{O}_S(D))=0$, otherwise $D$ is an \textit{irregular divisor}. The \textit{N\'eron-Severi group} $\mathrm{NS}(S)$ is the quotient group of the group of Cartier divisors on $S$ modulo numerical equivalence. It is well-known that $\mathrm{NS}(S)$ is a free finitely generated $\mathbb{Z}$-module and that coincides with the Picard group of $S$. The \textit{Picard number of $S$}
is the rank of the N\'eron-Severi group. We denote by $\mathcal{K}_S$ the class of a canonical divisor on $\mathrm{NS}(S)$ and will be called the \textit{canonical class}. The surface $S$ is \textit{anticanonical} if there exists an effective anticanonical divisor on $S$.

The strategy in this case is to divide the study depending on the self-intersection of $K_S$. The first case to consider is when $K_S^2>0$ and here we have a positive answer for the HH and AOP conditions.

\begin{proposition}\label{k2pos}
Let $S$ be a rational surface such that $K_S^2>0$. Then, $S$ is a HH surface and satisfies the AOP condition. 
\end{proposition}
\begin{proof}
Let $D$ be a nef divisor on $S$ such that $-K_S\cdot D=0$. Since $D$ is nef, then $D^2\geq 0$. Hodge Index theorem implies that $D^2\leq 0$, moreover, since $D^2=0$ then $D$ is numerically trivial. Therefore, $D=0$. This proves that $S$ satisfies the AOP condition.

Now, let $F$ be a nef divisor on $S$. Since $S$ is rational and $K_S^2>0$ we have that $F$ is also effective. If $-K_S\cdot F=0$, by the above argument we have that $F=0$ and then $h^1(S,\mathcal{O}_S(F))=h^1(S,\mathcal{O}_S)=0$ since $S$ is regular. If $-K_S\cdot F>0$, then \cite[Theorem III.1 (a) and (b)]{Har97} imply that $h^1(S,\mathcal{O}_S(F))=0$. Thus, we conclude that $S$ is a Harbourne-Hirschowitz surface.
\end{proof}

Next, we consider the case when $K_S^2=0$. For this case, the answer depends entirely on the multiples of $-K_S$, in fact, Harbourne proved in \cite[Corollary 10]{Har96} the following result for the HH condition and the same argument can be applied to obtain an answer for the AOP condition. 

\begin{theorem}[Harbourne]\label{k2zero}
Let $S$ be a rational surface such that $K_S^2=0$. The following statements are equivalent:
\begin{enumerate}
    \item $S$ does not satisfy the HH condition (respectively, does not satisfy the AOP condition). 
    \item There exists $r>0$ such that $-rK_S$ is an irregular nef divisor (respectively, is a nef divisor).
\end{enumerate}
\end{theorem}
\begin{proof} 
Since $S$ is rational and $K_S^2=0$, then $S$ is anticanonical, moreover, we have that every nef class also is an effective class. Let $F$ be a nef divisor on $S$. By \cite[Theorem III (a) and (b)]{Har97}, we have that if $-K_S\cdot F>0$, then $h^1(S,\mathcal{O}_S(F))=0$. So, the only possibility for a nef class to be irregular is that it is orthogonal to $-K_S$. Also, since $K_S^2=0$, by \cite[Lemma II.4]{Har97} we have that $-K_S\cdot F=0$ and $F^2=0$ imply that $F$ is a multiple of $-K_S$. 
The result follows from these facts.
\end{proof}

In particular, this result gives a criterion to ensure the HH and AOP conditions when $K^2=0$.

\begin{example}
Let $S$ be an elliptic rational surface, that is, there exists a morphism  $\rho:S\rightarrow\mathbb{P}^1$ whose general fiber is a smooth curve of genus one. In this case, $S$ is never a Harbourne-Hirschowitz surface nor satisfies the AOP condition. Indeed, if $D$ is a nef divisor on $S$ such that $-K_S.D=0$, then there exists a positive integer $n$ such that $D\sim -nK_S$ and  $h^1(S,\mathcal{O}_S(D))=n$, see \cite[Proposition 1.2]{AGL16}.
\end{example}

Finally, the last case to consider is when $K_S^2<0$. In such situation, the surface can be anticanonical or not, we focus our attention in the anticanonical case. Even in such case, both HH and AOP conditions may fail in general.

\begin{example}
Let $\pi:\widetilde{S}\rightarrow S$ be the blow-up of a rational elliptic surface $S$ at one point. Consider a nef divisor $D$ on $S$ which is orthogonal to $-K_S$. As we pointed out in the previous example, there exists a positive integer such that $D\sim -nK_S $ and $h^1(S,\mathcal{O}_S(D))=n$. The pull-back $\pi^*(D)$ of $D$ is a nef divisor on $\widetilde{S}$ and satisfies $h^1(\widetilde{S},\mathcal{O}_{\widetilde{S}}(\pi^*(D)))=h^1(S,\mathcal{O}_S(D))=n$. So, $\widetilde{S}$ is not a HH surface.
\end{example}

\begin{example}\label{threelines}
Consider the following configuration on $\mathbb{P}^2$. Let $L_p$, $L_q$ and $L_r$ be three lines that are not concurrent at one point. Consider points $p_1,p_2,p_3\in L_p$, $q_1,q_2,q_3\in L_q$ and $r_1,r_2,r_3,r_4\in L_r$ such that the points are different from the intersection points of the lines (see Figure \ref{10pts}). Let $T$ be the blow-up of $\mathbb{P}^2$ at these 10 points. 
\begin{figure}[h!]
 \centering
 \includegraphics[scale=.5]{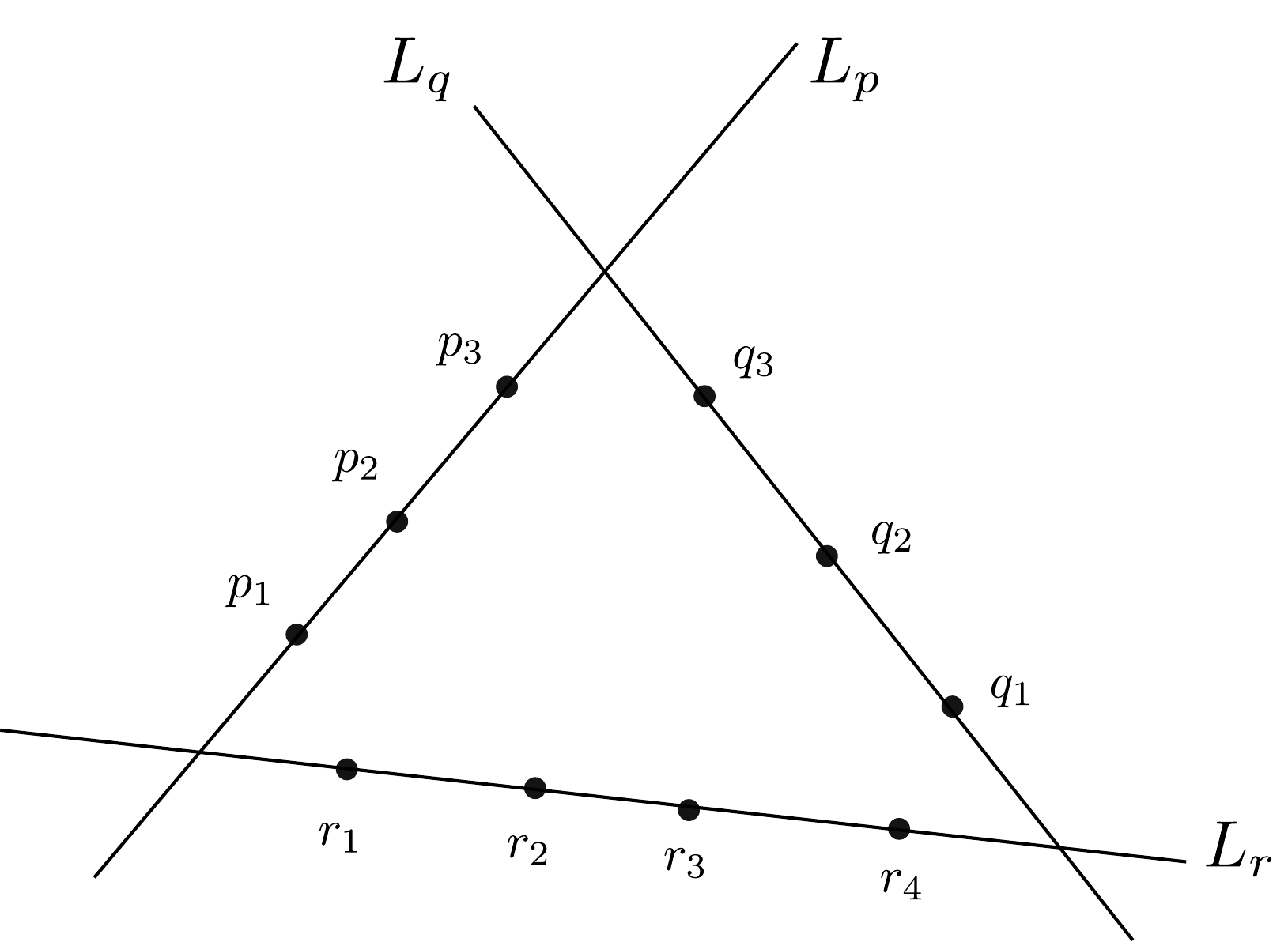}
 \caption{Configuration of points which defines the surface $T$.} \label{10pts}
\end{figure}

We claim that $T$ does not satisfy the AOP condition. In fact, consider the divisor class
\begin{equation*}
    \mathcal{D}=4\mathcal{H}-\mathcal{E}_{p_1}-\mathcal{E}_{p_2}-2\mathcal{E}_{p_3}-\mathcal{E}_{q_1}-\mathcal{E}_{q_2}-2\mathcal{E}_{q_3}-\mathcal{E}_{r_1}-\mathcal{E}_{r_2}-\mathcal{E}_{r_3}-\mathcal{E}_{r_4}.
\end{equation*}
This class corresponds to the strict transform of the divisor $L_p+L_q+L_r+L_{p_3q_3}$ on $\mathbb{P}^2$, where $L_{p_3q_3}$ denotes the line passing through $p_3$ and $q_3$. This is a nef class since the intersection of $\mathcal{D}$ with each of its irreducible components is zero and it is orthogonal to $-\mathcal{K}_T$. 
\end{example}

On the other hand, we have the following criterion proved in \cite[Theorem 2.5]{FL18} which relates the AOP and HH conditions: 

\begin{proposition}
Let $S$ be an anticanonical rational surface. If $S$ satisfies the AOP condition, then $S$ is a HH surface.
\end{proposition}

The key result used in the latter is due to Harbourne \cite[Theorem III.1]{Har97} in which he proved that if a nef class intersects the anticanonical class non-negatively, then the dimension of its first cohomology group is equal to zero. 

So, it remains to determine if the other implication holds true, that is, if $S$ is a HH surface, then it satisfies the AOP condition. In an equivalent way, we want to prove that if $S$ does not satisfy the AOP condition, then it is not a HH surface. In fact, since on an anticanonical rational surface any nef divisor intersecting an anticanonical divisor positively is a regular one, then the only case in which the HH condition could fail is when a non-trivial nef divisor is orthogonal to an anticanonical divisor. 

In one of the results of \cite{Har97}, Harbourne studied the properties of nef classes depending on their intersection with the anticanonical class. Particularly, he obtained the following for those nef classes that are orthogonal to the anticanonical class:

\begin{theorem}[Harbourne]\label{hariii.1}
Let $S$ be an anticanonical rational surface, $\mathcal{F}$ be a nef divisor class and let $D$ be a non-zero section of $-\mathcal{K}_S$. Consider $\mathcal{F}=\mathcal{H}+\mathcal{N}$, where $\mathcal{H}$ and $\mathcal{N}$ are the classes of the free part and the fixed part of $\mathcal{F}$ respectively. If $-\mathcal{K}_S\cdot \mathcal{F}=0$, then either:
\begin{enumerate}
    \item $\mathcal{N}=0$ in which case the sections of $\mathcal{F}$ are base point free, $\mathcal{F}\otimes\mathcal{O}_D$ is trivial and either
    \begin{enumerate}
        \item[(1.1)] $\mathcal{F}^2>0$ and $h^1(S,\mathcal{O}_S(\mathcal{F}))=1$, or
        \item[(1.2)] $\mathcal{F}=r\mathcal{C}$ and $h^1(S,\mathcal{O}_S(\mathcal{F}))=r$, where $r>0$ and $\mathcal{C}$ is a class of self-intersection equal to zero whose general section is reduced and irreducible.
    \end{enumerate}
    \item $\mathcal{N}$ is a smooth rational curve of self-intersection $-2$ in which case $h^1(S,\mathcal{O}_S(\mathcal{F}))=1$, $\mathcal{N}\otimes\mathcal{O}_D$ is trivial, and $\mathcal{H}=r\mathcal{C}$, where $r>1$ and $\mathcal{C}$ is reduced and irreducible with $\mathcal{C}^2=0$, $\mathcal{C}\cdot \mathcal{N}=1$ and $\mathcal{C}\otimes\mathcal{O}_D$ being trivial. 
    \item $\mathcal{N}+\mathcal{K}_S$ is an effective class.
\end{enumerate}
Moreover, (3) occurs if and only if $\mathcal{F}\cdot D=0$ but $\mathcal{F}\otimes\mathcal{O}_D$ is non-trivial. In this case, there exists a birational morphism of $S$ to a smooth projective anticanonical rational surface $T$ where one of the following occurs:
\begin{enumerate}
    \item[(a)] $\mathcal{K}_T^2<0$, there is a nef class $\mathcal{F}'$ on $T$, $\mathcal{F}$ is the pull-back of $\mathcal{F}'-\mathcal{K}_T$ and $h^1(S,\mathcal{O}_S(\mathcal{F}))=h^1(S,\mathcal{O}_T(\mathcal{F}'))=0$, or
    \item[(b)] $\mathcal{K}_T^2=0$, $\mathcal{H}$ and $\mathcal{N}$ are the pull-backs of $-s\mathcal{K}_T$ and $-r\mathcal{K}_T$ for some integers $s\geq 0$ and $r>0$ respectively, and $h^1(S,\mathcal{O}_S(\mathcal{F}))=\sigma$, where $\sigma=0$ if $s=0$, and otherwise $r<\tau$ and $\sigma=s/\tau$, where $\tau$ is the least positive integer such that the restriction of $-\tau \mathcal{K}_S$ to $D$ is trivial. 
\end{enumerate}
\end{theorem}

The last result implies that if $S$ is an anticanonical rational surface which does not satisfy the AOP condition, then there are cases when the regularity of nef divisors may fail but also there are cases when the regularity may hold true. So, in the following result we study the HH anticanonical rational surfaces for which the AOP condition fails. First, we need the following lemma:

\begin{lemma}\label{casek2neg}
Let $\pi:S\rightarrow T$ be a birational morphism where $S$ and $T$ are anticanonical rational surfaces and $\mathcal{K}_T^2<0$. Let $\mathcal{F}$ be a non-zero nef class on $S$ such that $-\mathcal{K}_S\cdot \mathcal{F}=0$. Assume that there exists a nef class $\mathcal{F}'$ on $T$ such that $\mathcal{F}=\pi^*(\mathcal{F}'-\mathcal{K}_T)$. Then $T$ does not satisfy the AOP condition.
\end{lemma}
\begin{proof}
Since $\pi:S\rightarrow T$ is a birational morphism, then $\pi$ can be written as a composition of a finite number of blow-ups, namely,
\begin{equation*}
 \xymatrix{
 S =T_n \ar@{->}@/_{7mm}/[rrrrr]_{\pi} \ar[r]^{\pi_n} & T_{n-1} \ar[r]^{\pi_{n-1}} &T_{n-2} \ar[r]^{\pi_{n-2}}
 &\cdots \ar[r]^{\pi_{2}} & T_{1}\ar[r]^{\pi_1} & T_0=T,
 }
\end{equation*}
where for every $i=1,\dots,n$, $\pi_i:T_{i}\rightarrow T_{i-1}$ is a blow-up at one point and $T_{i-1}$ is an anticanonical rational surface such that $\mathcal{K}_{T_{i-1}}^2<0$. 
Then, the equation $\mathcal{F}=\pi^*(\mathcal{F}'-\mathcal{K}_T)$ can be written as
\begin{equation}\label{fpullbacks}
    \mathcal{F}=\pi_n^*\circ\pi_{n-1}^*\circ\cdots\circ\pi_{2}^*\circ\pi_1^*(\mathcal{F}'-\mathcal{K}_T).
\end{equation}
We will prove that the divisor class $\mathcal{F}'-\mathcal{K}_T$ implies the failure of the AOP condition on $T$.

$\bullet$ $\mathcal{F}'-\mathcal{K}_T$ is a non-zero nef class. If $\mathcal{F}'-\mathcal{K}_T=0$, we would have that $\mathcal{F}'=\mathcal{K}_T$ but this is impossible because $\mathcal{F}'$ is a nef class. Now, consider an effective class $\mathcal{C}$ on $T$. Since a blow-up preserves the effectiveness of a divisor class, then it follows that \begin{equation*}
    \pi_i^*\circ\pi_{i-1}^*\circ\cdots
    \circ\pi_2^*\circ\pi_1^*(\mathcal{C})
\end{equation*}
is an effective class on $T_i$ for every $i=1,\dots,n$. Using the fact that the intersection number is preserved by pull-backs of a blow-up and Equation \eqref{fpullbacks}, we have that
\begin{align*}
\mathcal{C}\cdot (\mathcal{F}'-\mathcal{K}_T) 
    &= \left(\pi_1^*(\mathcal{C})\right)\cdot \left(\pi_1^*(\mathcal{F}'-\mathcal{K}_T)\right) \\
    &= \left(\pi_2^*\circ\pi_1^*(\mathcal{C})\right)\cdot 
    \left(\pi_2^*\circ\pi_1^*(\mathcal{F}'-\mathcal{K}_T)\right) \\
    &\vdots \\
    &= \left(\pi_n^*\circ\cdots\circ\pi_1^*(\mathcal{C}) \right)\cdot  \left( \pi_n^*\circ\cdots\circ\pi_1^*(\mathcal{F}'-\mathcal{K}_T) \right)\\ 
    &=\left(\pi_n^*\circ\cdots\circ\pi_1^*(\mathcal{C}) \right)\cdot \mathcal{F},  
    \end{align*}
and the last quantity is non-negative because $\mathcal{F}$ is a nef divisor class. Thus, $\mathcal{F}'-\mathcal{K}_T$ is a non-zero nef class on $T$.

$\bullet$ $\mathcal{F}'-\mathcal{K}_T$ is orthogonal to $-\mathcal{K}_T$. Denote by $\mathcal{E}_i$ the class of the exceptional divisor of the blow-up $\pi_i:T_i\rightarrow T_{i-1}$ for $i=1,\dots,n$. The idea to prove that $-\mathcal{K}_{T}.(\mathcal{F}'-\mathcal{K}_T)=0$ is to use recursively that the intersection number is preserved under the pull-back of a blow-up. First, lets note that 
\begin{align*}
    -\mathcal{K}_T\cdot (\mathcal{F}'-\mathcal{K}_T) 
    &=\pi_1^*(-\mathcal{K}_T)\cdot \pi_1^*(\mathcal{F}'-\mathcal{K}_T)\\
    &= (-\mathcal{K}_{T_1}+\mathcal{E}_1)\cdot  \left(\pi_1^*(\mathcal{F}'-\mathcal{K}_T)\right) \\
    &= \left(-\mathcal{K}_{T_1}\cdot \pi_1^*(\mathcal{F}'-\mathcal{K}_T)\right) + \left(\mathcal{E}_1\cdot \pi_1^*(\mathcal{F}'-\mathcal{K}_T)\right) \\
    &=-\mathcal{K}_{T_1}\cdot \pi_1^*(\mathcal{F}'-\mathcal{K}_T).
\end{align*}
Now, we can calculate the intersection number $-\mathcal{K}_{T_1}\cdot \pi_1^*(\mathcal{F}'-\mathcal{K}_T)$ using the second blow-up:
\begin{align*}
    -\mathcal{K}_{T_1}\cdot \pi_1^*(\mathcal{F}'-\mathcal{K}_T) 
    &=\left(\pi_2^*(-\mathcal{K}_{T_1})\right)\cdot \left(\pi_2^*\circ\pi_1^*(\mathcal{F}'-\mathcal{K}_T)\right)\\
    &= (-\mathcal{K}_{T_2}+\mathcal{E}_2)\cdot  \left(\pi_2^*\circ\pi_1^*(\mathcal{F}'-\mathcal{K}_T)\right) \\
    &= \left(-\mathcal{K}_{T_2}\cdot \pi_2^*\circ\pi_1^*(\mathcal{F}'-\mathcal{K}_T)\right) + \left(\mathcal{E}_2\cdot \pi_2^*\circ\pi_1^*(\mathcal{F}'-\mathcal{K}_T)\right) \\
    &=-\mathcal{K}_{T_2}\cdot \pi_2^*\circ\pi_1^*(\mathcal{F}'-\mathcal{K}_T).
\end{align*}
So, we have that 
\begin{equation*}
    -\mathcal{K}_T\cdot (\mathcal{F}'-\mathcal{K}_T) 
    = -\mathcal{K}_{T_1}\cdot \pi_1^*(\mathcal{F}'-\mathcal{K}_T) 
    = -\mathcal{K}_{T_2}\cdot \pi_2^*\circ\pi_1^*(\mathcal{F}'-\mathcal{K}_T).
\end{equation*}
Using recursively this argument, we have that for every $i=1,\dots,n-1$, it occurs that 
\begin{equation*}
    -\mathcal{K}_{T_{i}}\cdot \pi_i^*\circ\cdots\circ\pi_1^*(\mathcal{F}'-\mathcal{K}_T)=-\mathcal{K}_{T_{i+1}}\cdot \pi_{i+1}^*\circ\pi_i^*\circ\cdots\circ\pi_1^*(\mathcal{F}'-\mathcal{K}_T), 
\end{equation*}
and consequently
\begin{equation*}
    -\mathcal{K}_T\cdot (\mathcal{F}'-\mathcal{K}_T) 
    = -\mathcal{K}_{T_1}\cdot \pi_1^*(\mathcal{F}'-\mathcal{K}_T) 
    =\cdots
    = -\mathcal{K}_{T_n}\cdot \pi_n^*\circ\cdots\circ\pi_1^*(\mathcal{F}'-\mathcal{K}_T)
    = -\mathcal{K}_S\cdot \mathcal{F}=0.
\end{equation*}
Therefore, we have the existence of a non-zero nef class on $T$ that is orthogonal to $-\mathcal{K}_T$. We conclude that $T$ does not satisfy the AOP condition.
\end{proof}

\begin{remark}
The conditions of the previous lemma does not characterize the surfaces that does not satisfy the AOP condition. In fact, consider the configuration of Example~\ref{threelines} that defines the surface $T$ and add another point $r_5$ on the line $L_r$. Let $S$ be the surface obtained from the blow-up of $\mathbb{P}^2$ at this configuration of $11$ points. So, we have a birational morphism $\pi:S\rightarrow T$, both surfaces are anticanonical and $\mathcal{K}_T^2<0$. We have that the class $\mathcal{D}$ on $T$ is a non-zero nef class that is orthogonal to $-\mathcal{K}_T$. Thus, we have that $\mathcal{F}=\pi^*(\mathcal{D})$ is a non-zero nef class on $S$ such that $-\mathcal{K}_S\cdot \mathcal{F}=0$, this implies that $S$ does not satisfy AOP. However, in this case we have that $\mathcal{D}=\widetilde{L_{p_3q_3}}-\mathcal{K}_T$ and $\widetilde{L_{p_3q_3}}$ is the class of a $(-1)$-curve, which is not a nef class.
\end{remark}

Now, we present our main result for anticanonical rational surfaces that satisfy the Harbourne-Hirschowitz condition but that do not satisfy the Anticanonical Orthogonal Property.

\begin{theorem}
Let $S$ be a HH anticanonical rational surface with $\mathcal{K}_S^2<0$. If $S$ does not satisfy the AOP condition, then there exists a birational morphism $\pi:S\rightarrow T$ where $T$ is a HH anticanonical rational surface with $\mathcal{K}_T^2=0$ and does not satisfy the AOP condition.
\end{theorem}
\begin{proof}
Since $S$ does not satisfies the AOP condition, there exists a non-zero nef divisor $F$ on $S$ such that $-K_S.F=0$. Since in an anticanonical rational surface every nef divisor is effective and $S$ is a HH surface, we have that $h^1(S,\mathcal{O}_S(F))=0$. Consider the class $\mathcal{F}$ of $F$ and take $\mathcal{F}=\mathcal{H}+\mathcal{N}$, where $\mathcal{H}$ is the class of the free part of $\mathcal{F}$ and $\mathcal{N}$ the class of the fixed part. Using Theorem \ref{hariii.1} we have that cases (1) and (2) cannot happen, then the only possibility is case (3), i.e. $\mathcal{N}+\mathcal{K}_S$ is an effective class. Consequently, there exists a birational morphism $\pi:S\rightarrow T$ where $T$ is an anticanonical rational surface and by Proposition \ref{contrachh} we have that $T$ is also a HH surface. According to Theorem \ref{hariii.1}, there exist two possible cases:

\textbf{Case (a):} $\mathcal{K}_T^2<0$. In such case, there exists a nef class $\mathcal{F}'$ in $T$ such that $\pi^*(\mathcal{F}'-\mathcal{K}_T)=\mathcal{F}$. Moreover, we have that $h^1(T,\mathcal{O}_T(\mathcal{F}'))=h^1(S,\mathcal{O}_S(\mathcal{F}))=0$, but this is a consequence of the hypothesis $S$ is HH. Using Lemma \ref{casek2neg}, we have that $T$ does not satisfies the AOP condition. In addition, we have that the Picard number of $T$ is smaller than the Picard number of $S$ since $\pi$ is a composition of blow-ups. So, we have that $T$ is a HH anticanonical rational surface, $\mathcal{K}_T^2<0$ and $T$ does not satisfy the AOP condition, that is, we are in the initial setting. Also, note that $\mathcal{K}_S^2<\mathcal{K}_T^2<0$. Thus, we can apply the same argument and since the rank of the Picard group is finite, at some point we have to reach the second possibility of Theorem \ref{hariii.1}.

\textbf{Case (b):} $\mathcal{K}_T^2=0$. In this case, we have the existence of integer numbers $s\geq 0$ and $r>0$ such that $\mathcal{H}=\pi^*(-s\mathcal{K}_T)$ and $\mathcal{N}=\pi^*(-r\mathcal{K}_T)$. Note that since $h^1(S,\mathcal{O}_S(F))=0$ it occurs that $s=0$ and then $\mathcal{F}=\pi^*(-r\mathcal{K}_T)$. In addition, we have that $-r\mathcal{K}_T$ is a non-zero nef class: if $\mathcal{C}$ is the class of an effective divisor, then
\begin{equation*}
    -r\mathcal{K}_T\cdot \mathcal{C}=\pi^*(-r\mathcal{K}_T)\cdot \pi^*(\mathcal{C})
    =\mathcal{F}\cdot \pi^*(\mathcal{C})\geq 0
\end{equation*}
since $\pi^*(\mathcal{C})$ is an effective class and $\mathcal{F}$ is a nef class on $S$. Finally, the condition $\mathcal{K}_T^2=0$ implies that $-r\mathcal{K}_T$ is a non-zero nef divisor that is orthogonal to $-\mathcal{K}_T$ and we conclude the result. 
\end{proof}

As we mention in the introduction, the Harbourne-Hirschowitz conjecture, or Segre-Harbourne-Gimigliano-Hirschowitz conjecture, is an open problem of great interest nowadays. Because of this, the classification of Harbourne-Hirschowitz anticanonical rational surfaces remains open. To the best of our knowledge, the known examples of anticanonical rational surfaces that are Harbourne-Hirschowitz also satisfy the AOP condition. Examples of rational surfaces that satisfy the HH condition but not the AOP condition will be part of a future research project. 
\bigskip

\textbf{Acknowledgements} We would like to thank Mustapha Lahyane for his comments to improve this work. Also, we thank the anonymous referees for their suggestions and remarks to improve the quality of this paper.


\end{document}